\documentclass[12pt]{amsart}
\usepackage{amscd}
\usepackage{amssymb}
\usepackage[initials]{amsrefs}
\usepackage{enumitem}
\usepackage{verbatim}
\usepackage[mathscr]{eucal}	
\DeclareSymbolFont{pssymbols}     {OMS}{ztmcm}{m}{n}
\DeclareSymbolFontAlphabet{\mathpsscr}   {pssymbols}

\usepackage[all]{xy}
\SelectTips{cm}{}
\UseTips
\CompileMatrices

\newdir^{ (}{{}*!/-5pt/@^{(}}
\newdir_{ (}{{}*!/-5pt/@_{(}}

\theoremstyle{plain}%
\newtheorem{thm}{Theorem}[section]
\newtheorem{prop}[thm]{Proposition}
\newtheorem{lem}[thm]{Lemma}

\newtheorem*{thm*}{Theorem}

\theoremstyle{definition}%

\theoremstyle{remark}%

\newtheorem*{rems*}{Remarks}
\newtheorem*{notation}{Notation}
\newtheorem*{warning}{Warning}

\setlist[enumerate,1]{label=\textup{(\roman*)}}
\setlist[enumerate]{font=\normalfont}
\newlength{\myleftmargin}
\setlist[enumerate]{leftmargin=\myleftmargin}

\numberwithin{equation}{section}
\newcommand{\CC}{{\mathbf C}}
\newcommand{\RR}{{\mathbf R}}
\newcommand{\QQ}{{\mathbf Q}}
\newcommand{\ZZ}{{\mathbf Z}}

\newcommand{\EE}{{\mathbf E}}
\newcommand{\HH}{{\mathbf H}}
\newcommand{\Cat}{{\operatorname{\mathbf C}}}

\newcommand{\Mor}{\operatorname{Mor}}

\newcommand{\Sh}{\operatorname{\mathfrak S\mathfrak h}}
\newcommand{\Derived}{{\operatorname{\mathbf D}}}
\newcommand{\Derc}{\Derived_c}

\let\Perv\Perverse

\let\D\Dmodule
\renewcommand{\i}{i}		

\DeclareMathOperator{\id}{id}
\let\Sh\Sheaf
\renewcommand{\P}{\mathcal P}
\newcommand{\C}{{\mathcal C}}
\newcommand{\I}{\mathcal I}

\newcommand{\lk}{\kappa}       
\DeclareMathOperator{\mS}{SS}

\newcommand{\Xhat}{\widehat{X}}
\DeclareMathOperator{\Lk}{Lk}
\DeclareMathOperator{\Sp}{Sp}

\newcommand{\al}{\alpha}
\renewcommand{\gg}{{\mathfrak g}}
\renewcommand{\ll}{{\mathfrak l}}
\newcommand{\nn}{{\mathfrak n}}
\newcommand{\X}{\mathcal X} 
\newcommand{\Ext}{\operatorname{Ext}}
\newcommand{\codim}{\operatorname{codim}}
\newcommand{\pbar}{{\bar p}}    
\newcommand{\lsp}[1]{{}^{#1}}
\newcommand{\p}{\lsp{p}}
\newcommand{\n}[1]{\nobreakdash#1\hspace{0pt}}
\renewcommand{\t}[1]{$t$\nobreakdash#1\hspace{0pt}}
\renewcommand{\l}{\ell}
\DeclareMathOperator{\depth}{depth}

\begin{document}

\title[Perverse Sheaves]{Perverse sheaves and the reductive Borel-Serre
compactification}
\author{Leslie Saper}
\dedicatory{Dedicated to Steve Zucker, with great respect and admiration}
\address{Department of Mathematics\\Duke University\\Durham, NC 27708}
\email{saper@math.duke.edu}
\date{\today}

\maketitle

\section{Introduction}

This paper is a report on work in progress to better understand the
category of perverse sheaves on the Baily-Borel compactification $X^*$ of a
locally symmetric variety by using
perverse sheaves on the reductive Borel-Serre compactification $\Xhat$.

\subsection{Definition}
Perverse sheaves were introduced in \cite{refnBeilinsonBernsteinDeligne} by
Beilinson, Bernstein, and Deligne following Goresky and MacPherson's
introduction of intersection homology \cite{refnGoreskyMacPhersonIHOne}.
(For a full account of the exciting discoveries and interactions around that
time, see Kleiman's excellent history
\cite{refnKleimanIntersectionHomology}).  We begin by briefly recalling the
definition (and to avoid confusion, we point out, as in
\cite{refnBeilinsonBernsteinDeligne}, that a perverse sheaf is not actually
a sheaf nor is it perverse).

Let $Y$ be a stratified pseudomanifold and fix an integer-valued function
on the set of strata of $Y$ called the \emph{perversity};%
\footnote{The origin of the term ``perversity'' goes back to the creation
  of intersection homology by Goresky and MacPherson
  \cite{refnGoreskyMacPhersonIHOne} where chains were constrained so that
  the dimension of their intersection with a singular stratum $S$ was at
  most $p(S)-p(S_0)-1$ more than that allowed by transversality ($S_0$
  being the open dense stratum).}
if all the strata of $Y$ have even dimension, it is most useful to take for
$p$ the \emph{middle perversity}, $p(S) = -(1/2)\dim S$.  Let $\Derc^b(Y)$
be the constructible bounded derived category of $Y$; objects may be
represented by complexes of sheaves on $Y$ (up to quasi-isomorphism) whose
cohomology lives in a bounded range of degrees and is locally constant
along each stratum.  By definition, a perverse sheaf $\P$ on $Y$
(constructible with respect to the given stratification) is an object $\Sh$
whose local cohomology $H(i_S^*\Sh)$ along each stratum $i_S\colon
S\hookrightarrow Y$ lives in degrees $\le p(S)$ and whose local cohomology
supported on $S$, $H(i_S^!\Sh)$, lives in degrees $\ge p(S)$.  The perverse
sheaves on $Y$ form an abelian category $\Perv(Y)$ which is Artinian and
Noetherian.  For every connected stratum $S$ and irreducible local system
$\EE$ on $S$ there is a simple perverse sheaf $\P_S(\EE)=\P_{p,S}(\EE)$ and
all simple objects are obtained in this fashion.  In fact $\P_S(\EE)$ is
the intersection cohomology sheaf of $\overline{S}$ with coefficients in
$\EE$ shifted by $-p(S)$.

When $Y$ is a complex algebraic variety, we will always consider the middle
perversity.  In addition, one usually does not fix the (algebraic)
stratification in the definitions of $\Derc^b(Y)$ and $\Perv(Y)$.

\subsection{Two original applications}
Perverse sheaves have played a critical role in the topology of algebraic
varieties and representation theory from their beginning; we
mention just two examples from the early days.  For a fuller survey of
applications, see \cite{refndeCataldoMiglioriniDecompositionTheorem}.

The first example is the Kazdhan-Lusztig conjecture.  Let $Y=G/B$ be the
flag variety associated to a simply connected semisimple complex algebraic
group $G$.  It is stratified by its $B$-orbits $S_w$ which are indexed by
$w$ in the Weyl group $W$ of $G$; each closure $\overline{S_w}$ is a
Schubert variety.  The correspondence $w \leftrightarrow S_w$ is order
preserving in the sense that $y\le w$ if and only if $S_y \subset
\overline{S_w}$.  For $y\le w\in W$, Kazdhan and Lusztig
\cite{refnKazhdanLusztigCoxeterGroups} gave a conjectural formula for the
multiplicity of the irreducible $\gg$\n-module $L_y$ with highest weight
$y\rho-\rho$ in the composition series of the Verma module $M_w$ with
highest weight $w\rho-\rho$.  The formula involved a combinatorially
defined polynomial $P_{y,w}(q)$ which appeared to be related to the failure
of Poincar\'e duality on $\overline{S_w}$; later they showed
\cite{refnKazhdanLusztigSchubertVarieties} that the coefficients in
$P_{y,w}(q)$ were $\dim H^k(i_{S_y}^*\P_{S_w}(\CC))$, the local Betti
numbers along $S_y$ of the simple perverse sheaf associated to $S_w$.

The conjecture was resolved independently by Beilinson and Bernstein
\cite{refnBeilinsonBernsteinKazhdan-LusztigConjecture} and by Brylinski and
Kashiwara \cite{refnBrylinskiKshiwaraKazhdan-LusztigConjecture} by
transferring the problem from representations to perverse sheaves.
Specifically one shows that a certain category of $\gg$\n-modules
(including all $L_y$ and $M_w$) is equivalent to the category of regular
holonomic $\D_Y$\n-modules and that by the Riemann-Hilbert
correspondence this category is equivalent to $\Perv(Y)$.


The second original example of the importance of perverse sheaves is the
decomposition theorem.  This deep result about the topology of proper 
algebraic maps was first proved by Beilinson, Bernstein, Deligne, and
Gabber \cite{refnBeilinsonBernsteinDeligne}*{\emph{Th\'eor\`eme} 6.2.5} and
later extended by Saito \citelist{\cite{refnSaitoModulesHodgePolarisables}
  \cite{refnSaito}}.  It says that if $X\to Y$ is a proper morphism of
algebraic varieties, then the direct image of a simple perverse sheaf of
Hodge type on $X$ decomposes into the direct sum of shifted simple perverse
sheaves, likewise of Hodge type, supported on subvarieties of $Y$.  Here by
a simple perverse sheaf of Hodge type we mean one that corresponds to a
polarizable Hodge module in the sense of Saito; by
\citelist{\cite{refnSaitoModulesHodgePolarisables}*{\emph{Lemmes} 5.1.10,
    5.2.12} \cite{refnSaito}*{Theorem 0.2}}, such simple perverse sheaves
are precisely those of the form $\P_S(\EE)$ where $\EE$ underlies a real
polarizable variation of Hodge structure.  As a special case, if
$\widetilde Y \to Y$ is a resolution of singularities, the theorem implies
that the ordinary cohomology $H^\cdot( \widetilde Y;\CC)$ is the direct sum
of shifted middle perversity intersection cohomology groups of subvarieties
of $Y$ with various local systems as coefficients.  Furthermore one summand
will be $IH^\cdot(Y;\CC)$, the intersection cohomology of $Y$ itself.

The proof in \cite{refnBeilinsonBernsteinDeligne} is for a simple perverse
sheaf of geometric origin and proceeds by first dealing with the theorem
over the algebraic closure of a finite field and then lifting the result to
$\CC$.  The version of the theorem stated above was proved by Saito
\citelist{\cite{refnSaitoModulesHodgePolarisables} \cite{refnSaito}} using
purely characteristic $0$ methods.  It uses mixed Hodge modules, which
are regular holonomic $\D$\n-modules equipped with various filtrations and
satisfying certain conditions; they correspond to the mixed perverse
sheaves over finite fields considered in
\cite{refnBeilinsonBernsteinDeligne}.  A pure Hodge module corresponds
under the Riemann-Hilbert correspondence to a simple perverse sheaf of
Hodge type.

More recently, other proofs and generalizations of the decomposition
theorem have been given.  A proof of the constant coefficient case using
classical Hodge theory was given by de Cataldo and Migliorini
\cite{refndeCataldoMiglioriniHodgeTheory}; for more details and a
discussion of the different approaches to the decomposition theorem see
also their survey \cite{refndeCataldoMiglioriniDecompositionTheorem}.
Kashiwara has conjectured \cite{refnKashiwaraSemisimpleHolonomic} that the
decomposition theorem should hold for \emph{any} simple perverse sheaf.
This conjecture has been proven analytically using polarizable pure twistor
$\D$\n-modules by work of Sabbah \cite{refnSabbahTwistorD-modules} and
Mochizuki \citelist{\cite{refnMochizukiTameHarmonicBundlesOne}
  \cite{refnMochizukiTameHarmonicBundlesTwo}}.  At about the same time, an
arithmetic proof was given independently by B\"ockle and Khare
\cite{refnBoeckleKhareArithmeticFundamentalGroups} and by Gaitsgory
\cite{refnGaitsgoryDeJongsConjecture}; they proved de Jong's conjecture
\cite{refndeJongArithmeticFundamentalGroups} which by Drinfeld
\cite{refnDrinfeldKashiwaraConjecture} implies Kashiwara's conjecture.
Note that Kashiwara actually conjectures the theorem should hold for a
simple holonomic $\D_X$\n-module with possibly irregular singularities;
this has been settled by Mochizuki \cite{refnMochizukiWildHarmonicBundles}.

\subsection{Locally symmetric varieties}
Our main interest here is the category of perverse sheaves on a locally
symmetric variety.  Let $X = \Gamma\backslash D = \Gamma \backslash G(\RR)
/ A_G(\RR) K$ be an arithmetic quotient of a symmetric space of noncompact
type; here $G$ is a reductive algebraic group defined over $\QQ$, $K$ a
maximal compact subgroup of $G(\RR)$, $A_G$ is the maximal $\QQ$\n-split
torus in the center of $G$, and $\Gamma$ is an arithmetic subgroup.  Our
main focus will be when $D$ is Hermitian symmetric and unless otherwise
noted we shall assume that in this introduction, however some non-Hermitian
symmetric spaces will arise as well.

A natural choice of compactification in the Hermitian case is the
Baily-Borel compactification $X^*$ \cite{refnBailyBorel} (topologically it
is one of Satake's compactifications).  The Baily-Borel compactification is
a projective algebraic variety defined over a number field; it is commonly
called a locally symmetric variety.  Locally symmetric varieties and
Shimura varieties, their adelic variants, play an important role in number
theory, automorphic forms, and the Langlands's program.

The Baily-Borel compactification $X^*$ has a natural stratification
$\coprod_{Q} F_Q$, where $Q$ ranges over all $\Gamma$\n-conjugacy classes
of saturated parabolic $\QQ$\n-subgroups of $G$; when $G$ is almost
$\QQ$\n-simple the saturated condition simply means $Q$ is  maximal parabolic
$\QQ$\n-subgroup.  Each $F_Q$ is again a Hermitian locally symmetric space
associated to a reductive group $L_{Q,h}$, a quotient of $Q$; its closure
$\overline{F_Q}$ in $X^*$ is a subvariety and is the Baily-Borel
compactification $F_Q^*$ of $F_Q$.  We fix this stratification when
considering perverse sheaves on $X^*$.

On the locally symmetric space $X$ one usually considers local systems
$\EE$ that arise from a representation $E$ of $G$ as opposed to merely a
representation of $\Gamma$.  We say
that the simple perverse sheaf $\P_X(\EE)$ on $X^*$ is \emph{reductively
  constructible} if the coefficient system $\EE$ arises from a
representation $E$ of $G$ as above; similarly $\P_{F_Q}(\EE)$ is
reductively constructible if $\EE$ arises from a representation of
$L_{Q,h}$.

Zucker's conjecture \cite{refnZuckerWarped} gives an analytic realization
of a simple reductively constructible perverse sheaf $\P_X(\EE)$.  The
choice of an admissible inner product \cite{refnMatsushimaMurakami} on $E$
induces a metric on $\EE$ which we assume fixed.  The conjecture then
states that there is a natural isomorphism $\P_X(\EE) \cong \mathscr
L_2(X^*;\EE)[-p(X)]$, where $\mathscr L_2(X^*;\EE)$ is the sheafification
of the presheaf that associates to $U\subset X^*$ the complex of
$\EE$\n-valued differential forms on $U\cap X$ which, together with their
exterior derivative, are $L^2$ with respect to the natural locally
symmetric metric on $X$ and the above metric on $\EE$.  The conjecture was
independently settled at about the same time by Looijenga
\cite{refnLooijenga} and Saper and Stern \cite{refnSaperSternTwo} using
quite different methods.

The conjecture thus shows that the reductively constructible simple
perverse sheaves $\P_{F_Q}(\EE)$ are related to representations
$L^2(\Gamma_{L_{Q,h}}\backslash L_{Q,h}(\RR))$ (see
\cite{refnBorelGarland}) and hence to automorphic forms.  The full category
of (not necessarily simple) perverse sheaves on $X^*$ is thus clearly worth
further study.

\subsection{The reductive Borel-Serre compactification}
\label{ssectRBS}
In general $X^*$ is very singular so it may be profitable to study its
perverse sheaves by using a resolution.

A natural choice for an algebraic resolution, though non-unique, is one of
the smooth toroidal compactifications $\tilde\pi\colon \widetilde X_\Sigma
\to X^*$ \cite{refnAshMumfordRapoportTai}.  Let $\P_X(\EE)$ be a reductively
constructible simple perverse sheaf on $X^*$.  Let $U =\tilde\pi^{-1}(X)
\subset \widetilde X_\Sigma$; via $\tilde\pi$ we can view $\EE$ also as a
local system on $U$.  The decomposition theorem applies%
\footnote{We only need the version proved by Saito here by the following
  argument.  In the Hermitian case, Zucker
  \cite{refnZuckerLocallyHomogeneous} has shown $\EE$ underlies a complex
  polarizable variation of Hodge structure; it only a real polarizable
  variation of Hodge structure if the representation $E$ is real.  However
  $\EE\otimes_{\RR} \CC$ does underlie a real polarizable variation of
  Hodge structure so Saito's decomposition theorem applies to it.  Now
  $\EE\otimes_{\RR} \CC \cong \EE \oplus \overline{\EE}$ and  one can prove
  that the decomposition theorem holds for a direct sum if and only if it
  holds for each summand.}
to show that $\tilde\pi_*\P_U(\EE)$ is the direct sum of (shifted) simple
perverse sheaves on $X^*$, one of which is $\P_X(\EE)$.

However this argument does not easily extend to treat simple perverse
sheaves supported on smaller closed strata of $X^*$.  For example 
$\tilde\pi^{-1}(F_{Q})$ is rarely smooth nor irreducible.

To get around this we look instead at the reductive Borel-Serre
compactification $\Xhat$.  This was introduced by Zucker in the same paper
\cite{refnZuckerWarped} where he made his conjecture about the
$L^2$\n-cohomology of $X$.  (It is actually defined for any arithmetic
locally symmetric space, not necessarily Hermitian.)  It is non-algebraic
(in fact it may have odd dimensional strata) but it is canonically
associated to $X$ and its singularities are easy to describe.  Zucker
\cite{refnZuckerSatakeCompactifications} showed that there is a continuous
quotient map $\pi\colon \Xhat \to X^*$ extending the identity on $X$ so
that $\Xhat$ may be viewed as a partial resolution of singularities.
Goresky and Tai \cite{refnGoreskyTaiToroidalReductiveBorelSerre} show that
up to homotopy any $\tilde\pi$ factors through $\Xhat$.  So in this sense
it is a minimal partial resolution.

The reductive Borel-Serre compactification $\Xhat$ has a natural
stratification $\coprod_{P} X_P$, where $P$ ranges over all
$\Gamma$\n-conjugacy classes of parabolic $\QQ$\n-subgroups of $G$.  Each
$X_P$ is again an arithmetic locally symmetric space associated to a
reductive group, namely $L_P$, the Levi quotient of $P$; its closure
$\overline{X_P}$ in $\Xhat$ is the reductive Borel-Serre compactification
$\Xhat_P$ of $X_P$.  We fix this stratification when considering perverse
sheaves on $\Xhat$.  Since $\Xhat$ may have odd dimensional strata, there
are two middle perversities, $p_-$ and $p_+$ (see
\eqref{eqnMiddlePerversities}), that we will consider, both depending only
on the dimension of the stratum.  A simple perverse sheaf $\P_{X_P}(\EE)$
on $\Xhat$ is \emph{reductively constructible} if $\EE$ is induced from a
representation $E$ of $L_P$.

Despite $\Xhat$ not being algebraic, it ``wants'' to be algebraic.  For
example, Zucker showed \cite{refnZuckerRBSThree} that its cohomology
carries a mixed Hodge structure such that
\[
H^\cdot(X^*) \xrightarrow{\pi^*} H^\cdot(\Xhat) \xrightarrow{i_X^*}
H^\cdot(X)
\]
are morphisms of mixed Hodge structures.  (Here the outer two cohomology
groups carry Deligne's canonical mixed Hodge structure.)  Furthermore,
Ayoub and Zucker \cite{refnAyoubZuckerRBSMotive} construct a motive
corresponding to $\Xhat$.

As another example of the algebraic-like nature of $\Xhat$ is the
conjecture of Rapoport \cite{refnRapoport} and of Goresky and MacPherson
\cite{refnGoreskyMacPhersonWeighted}.  This conjecture was proved by Saper
and Stern \cite{refnRapoport}*{Appendix} when the  $\QQ$\n-rank of $G$ is
$1$, and by Saper \cite{refnSaperLModules} in general.  It states that for
a local system $\EE$ associated to a representation of $G$, 
\begin{equation}
  \label{eqnRapoportConj}
   \pi_* \P_{X_G}(\EE) = \P_{F_G}(\EE)\ ;
\end{equation}
Here the perversity on the left (for $\Xhat$) can be either of the two
middle perversities---one obtains the same pushforward.

Unlike in the situation for $\tilde\pi\colon \widetilde X_\Sigma \to X^*$,
the methods of \cite{refnSaperLModules} can be used to generalize
\eqref{eqnRapoportConj} to all reductively constructible simple perverse
sheaves $\P_{X_P}(\EE)$ on $\Xhat$.  Specifically (see
\cite{refnSaperLModules}*{\S21}) for every parabolic $\QQ$\n-subgroup $P$,
there is a saturated parabolic $\QQ$\n-subgroup $P^\dag$ so that $\pi(X_P)
= F_{P^\dag}$.  In fact, $\pi|_{X_P}\colon X_P \to F_{P^\dag}$ is a flat
bundle which becomes trivial over a finite cover of $F_{P^\dag}$; the fiber
is $X_{P,\ell}$, a locally symmetric space \emph{not} usually of Hermitian
type.  We have the following extension of the decomposition theorem despite
$\Xhat$ not being algebraic:
\begin{thm*}[Decomposition theorem for $\pi\colon \Xhat \to X^*$]
  Let $\EE$ be an irreducible local system on a stratum $X_P$ of $\Xhat$
  which is induced from an algebraic representation of $L_P$.  Let $p$ be a
  middle perversity.  Then
  \begin{equation}
  \pi_* \P_{p,X_P}(\EE) = \bigoplus_i
  \P_{F_{P^\dag}}( \HH^i(\Xhat_{P,\ell};\P_{p, X_{P,\ell}}(\EE) )) [-i] \ .
\end{equation}
\end{thm*}
Note that while the left-hand side a priori depends on the choice of middle
perversity, the right-hand side does not since $X^*$ has only
even-dimensional strata.  Also the theorem makes it clear that $\pi_*$ is not
injective on objects since the map $P \mapsto P^\dag$ is not injective.

Here is a sketch of the proof.  It suffices to
pass to a finite cover so one can arrange that $X_P = X_{P,\ell} \times
F_{P^\dag}$ whence $\Xhat_P = \Xhat_{P,\ell} \times \widehat F_{P^\dag}$,
with $\pi|_{\Xhat_P}$ being projection on the second factor.  Now apply the
K\"unneth formula of Cohen, Goresky, and Ji
\cite{refnCohenGoreskyJiKuenneth} (which one checks is applicable for
middle perversities) and apply Rapoport's conjecture to the second factor.

\subsection{Summary of this paper}
The theorem above shows that it is reasonable to study the category of
reductively constructible perverse sheaves on the Baily-Borel
compactification $X^*$ by studying the category of perverse sheaves on the
reductive Borel-Serre compactification $\Xhat$.  By the above version of
the decomposition theorem, we understand $\pi_*$ on simple objects; the
goal of this very modest paper is to begin to understand extensions better
in this context.  After recalling the basics of $t$\n-structures in
\S\ref{sectTstructures}, we carefully define middle perversity perverse
sheaves on a stratified pseudomanifold in \S\ref{sectPerverse}, playing
special attention to the issues that arise due to odd dimensional strata.
We also note certain non-trivial extensions that arise due to odd
codimension strata.  In \S\ref{sectExtensions} we indicate how one may
calculate extensions between two simple perverse sheaves.  After discussing
needed background on the reductive Borel-Serre compactification in
\S\ref{sectRBS} and the link cohomology of simple perverse sheaves in
\S\ref{sectLinkCohomology}, we conclude in \S\ref{sectComputations} by
doing the exercise of calculating all extensions between simple perverse
sheaves for $\widehat{\mathscr A_2}$, the reductive Borel-Serre
compactification of the moduli of principally polarized abelian surfaces.

\subsection{A potential application: perverse cohomology}
One reason to better understand extensions is to be able to calculate the
\emph{perverse cohomology} of an object $\Sh$ in $\Derived_{rc}^b(\Xhat)$
(or $\pi_* \Sh$ in  $\Derived_{rc}^b(X^*)$), 
the reductively constructible bounded derived category of sheaves.  Recall
that the definition of a perverse sheaf started with sheaves as the basic
objects, proceeded to the derived category, and within it found the abelian
category $\Perv(Y)$ of perverse sheaves.  Thus a perverse sheaf is
represented as a complex of ordinary sheaves.  However it has become clear
that it is also useful to view perverse sheaves themselves as the basic
objects and in some cases the roles of the two types of objects can be
reversed.  Beilinson \cite{refnBeilinsonDerivedCategory} shows that if $Y$
is an algebraic variety and we consider all algebraic stratifications (not
a fixed one), then the bounded derived category of $\Perv(Y)$ is equivalent
to $\Derc^b(Y)$ within which one finds the category of ordinary
constructible sheaves.  Thus in this setting, a constructible sheaf, or in
fact any object $\Sh$ of $\Derc^b(Y)$, can be represented as a complex of
perverse sheaves $\dots \to \P^{i-1} \to \P^i \to \P^{i+1} \to \dots$.  The
cohomology of this complex $H^i(\P^\cdot)$ is a perverse sheaf called the
perverse cohomology $\p H^i(\Sh)$.

Even when Beilinson's theorem does not apply, as it does not for $\Xhat$
since we have fixed a stratification, the perverse cohomology can be
defined as $\p H^i(\Sh) = \p\tau_{\le0} \p\tau_{\ge0} \Sh[i]$.
Furthermore, just like for the ordinary cohomology sheaves of a complex of
sheaves, there is a spectral sequence with $E_2^{ij} = H^j(Y;\p H^i(\Sh))$
that abuts to $H^{i+j}(Y; \Sh)$.  Many important invariants of $X$ are
realized as the cohomology of complexes of sheaves on $\Xhat$, for example,
the cohomology of the arithmetic group is $H^\cdot(\Gamma; E) =
H^\cdot(\Xhat; i_{X*}\EE)$.  Thus calculating $\p H^i(\Sh)$, or at least
having bounds on the degrees in which $H^\cdot(\Xhat;\p H^i(\Sh))$ can be
nonzero, is important.

One potential approach to calculating, or at least approximating, the perverse
cohomology of $\Sh$ is through its \emph{micro-support}.  In
\cite{refnSaperLModules}*{\S7}, we defined the micro-support $\mS(\Sh)$ and
proved a vanishing theorem \cite{refnSaperLModules}*{\S10} for
$H^\cdot(\Xhat;\Sh)$ based on $\mS(\Sh)$.  As an application, we were able
to prove \cite{refnSaperIHP}, for example, that in the Hermitian case,
$H^i(\Gamma;\EE)=0$ for $i < (1/2)\dim X$ provided $E$ had regular highest
weight; this result was independently proved by Li and Schwermer
\cite{refnLiSchwermer} by different methods.

We also determined \cite{refnSaperLModules}*{\S17} the micro-support of a
simple perverse sheaf $\P_{X_P}(\EE)$ provided the $\QQ$\n-root system of
$L_P$ did not have a factor of type $D_n$, $F_4$, or $E_n$ (a restriction
that should be removable).  The result shows that if a perverse sheaf is
semi-simple and satisfies a certain conjugate self-contragredient
assumption on the coefficients, the micro-support determines the perverse
sheaf.  In fact we will show elsewhere that under certain conditions, even
for a not necessarily semi-simple perverse sheaf, the micro-support
determines the simple constituents and thus potentially the micro-support
can be used as a tool to calculate the perverse cohomology.

\subsection{Acknowledgments}
I would like to thank Amnon Neeman for helpful discussions regarding
extensions.  I would also like to thank Lizhen Ji, Steve Zucker, and an
expert anonymous referee for thoughtful comments concerning this paper.
Needless to say, any failings of this paper are solely due to myself.

It was a great honor to be invited to speak at the birthday conference for
Steve Zucker and to contribute to this volume.  Steve has always been a
source of inspiration and encouragement to me.  When I was a graduate
student and he was a visitor at the Institute, I remember him sitting down
with me in the common room and carefully showing me the derivation of the
Poincar\'e punctured disk metric.  In later encounters he always took a
deep interest in my work which I appreciated.  I also enjoyed long evenings
with him spent listening to classical music as well as many delicious meals
(particularly Maryland crabs whacked open with a wooden mallet!).

\begin{notation}
  Morphisms in a category $\Cat$ will be denoted $\Mor_{\Cat}(A,B)$; when
  $\Cat$ is a category of representations of a group $G$, we simply write
  $\Mor_{G}(A,B)$.

 If $A$ is an object of some derived category of sheaves on a space $Y$ and
 $V\subset Y$ is an open subset, we will sometimes abuse notation and also
 denote by $A$ the  inverse image of $A$ to the derived category of
 sheaves on $V$, that is $A|_V$.

 If $S \subset Y$, we let $i_S\colon S \hookrightarrow Y$ denote the
 inclusion; we use the same notation to denote the inclusion into any
 subset of $Y$ containing $S$.  If $V$ is an open subset of $Y$ containing
 $S$ for which $S$ is closed in $V$, we let $j_S\colon U=V\setminus S
 \hookrightarrow V$.  We thus have inclusions of complementary open and
 closed subsets
\begin{equation}
\xymatrix{
{U} \ar@{^{ (}->}[r]^-{j_S} & {V} & {S\ .} \ar@{_{ (}->}[l]_-{i_S} 
}
\end{equation}
Often $S$ is a stratum of a stratification of $Y$ and $V$ is an open union
of strata for which $S$ is a minimal stratum.

If $S$, $S' \subset Y$, we write $S\preccurlyeq S'$ if $S \subseteq
\overline{S'}$ and $S\prec S'$ if $S \subsetneq \overline{S'}$.


\end{notation}

\section{$t$-structures}
\label{sectTstructures}
We briefly recall the theory of \t-structures on a  triangulated category
$\Derived$ following \cite{refnBeilinsonBernsteinDeligne}*{\S1.3}.

A \emph{\t-structure} on $\Derived$ consists of two full subcategories
$\Derived^{\le0}$ and $\Derived^{\ge0}$, closed under isomorphism,
satisfying the following 
conditions.  First set $\Derived^{\le k}= \Derived^{\le0}[-k]$ and
$\Derived^{\ge k}= \Derived^{\ge0}[-k]$.  We require
\begin{equation}\label{eqnTstructureOne}
  \text{$\Mor_{\Derived}(A,B)=0$ for $A\in \Derived^{\le 0}$, $B\in
    \Derived^{\ge 1}$}.
\end{equation}
Secondly we require that
\begin{equation}
\Derived^{\le0} \subseteq \Derived^{\le 1} \qquad \text{and}\qquad
\Derived^{\ge 0} \supseteq \Derived^{\ge 1}\ .
\end{equation}
And finally we assume that for any object $X\in \Derived$, there exists a
distinguished triangle
\begin{equation}
A\to X\to B \xrightarrow{[1]} \qquad  \text{($A\in
\Derived^{\le0}$, $B\in \Derived^{\ge 1}$).}
\end{equation}

In fact the distinguished triangle above is unique up to unique isomorphism
and one can use it to define truncation functors $\tau_{\le0}X = A$ and
$\tau_{\ge1} X = B$.  By shifting we obtain truncation functors $\tau_{\le
  k}\colon \Derived \to \Derived^{\le k}$ and $\tau_{\ge k}\colon \Derived
\to \Derived^{\ge k}$ for all $k\in \ZZ$.  There are natural morphisms
$\tau_{\le k} \to \id$ and $\id \to \tau_{\ge k}$, which induce adjoint
relations
\begin{align}
  \label{eqnRightTauAdjoint}
  \Mor_{\Derived}(A,B) &\cong \Mor_{\Derived^{\le k}}(A,\tau_{\le k}B) &&
  \text{($A\in \Derived^{\le k}$, $B\in \Derived$)} \\
  \label{eqnLeftTauAdjoint}
  \Mor_{\Derived}(A,B) &\cong \Mor_{\Derived^{\ge k}}(\tau_{\ge k}A,B) &&
\text{($A\in \Derived$, $B\in \Derived^{\ge k}$)}.
\end{align}

The \emph{heart} of a t-structure is $\Cat = \Derived^{\le0}\cap
\Derived^{\ge0}$.  It is a full abelian subcategory of $\Derived$.
Furthermore, a short exact sequence $0\to C' \to C \to C''\to 0$ in $\Cat$
corresponds to a distinguished triangle $C'\to C \to C'' \xrightarrow{[1]}$
in $\Derived$ with all objects belonging to $\Cat$ and vice-versa.  Thus
$\Ext^1_{\Cat}(C'',C') = \Mor_{\Derived}(C'',C'[1])$.

The functor $H^0\colon \Derived \to \Cat$ given by
$H^0(A)=\tau_{\ge0}\tau_{\le0} A = \tau_{\le0}\tau_{\ge0}A$ is a
cohomological functor.  We set $H^k(A) = H^0(A[k])$

\begin{lem}\label{lemMorphismsWithOneOverlap}
If $A\in \Derived^{\le k}$ and $B\in\Derived^{\ge k}$, then
$\Mor_\Derived(A,B) \cong \Mor_\Cat(H^k(A), H^k(B))$.
\end{lem}
\begin{proof}
\begin{align*}
  \Mor_\Derived(A,B) &= \Mor_\Derived(\tau_{\le k}A,B)
  &&\text{($A=\tau_{\le k}A$ since $A\in \Derived^{\le k}A$)} \\
   &= \Mor_{\Derived}(\tau_{\ge k}\tau_{\le k} A, B) &&\text{(by
    \eqref{eqnLeftTauAdjoint} since $B\in\Derived^{\ge k}$)} \\
  &= \Mor_{\Derived}(H^k(A)[-k],B)\ .
\end{align*}
Similarly one may replace $B$ here by $H^k(B)[-k]$ whence the
lemma since $\Cat$ is full.
\end{proof}

In this paper, $\Derived$ will be the constructible bounded derived category of
sheaves on a stratified pseudomanifold $Y$.  The \emph{standard \t-structure} has $\Derived^{\le0}$
consisting of objects $\Sh$ satisfying
\begin{equation}
H^k(\Sh) = 0 \text{ for $k>0$}\ ,
\end{equation}
and $\Derived^{\le0}$ consisting of objects $\Sh$ satisfying 
\begin{equation}
H^k(\Sh) = 0 \text{ for $k<0$}\ .
\end{equation}
In this case, the truncation functors are the usual truncations of a
complex, and objects in the heart may be represented by constructible
sheaves viewed as complexes with only one nonzero term in degree $0$.

In \S\ref{sectPerverse} we will define the perverse \t-structure whose
heart is the category of perverse sheaves.  The subcategories, truncation
functors, and cohomology functors for this \t-structure will always be
distinguished by a left superscript $p$, as in $\p \Derived^{\le0}$.

\section{Perverse sheaves}
\label{sectPerverse}
\subsection{Definition}
Let $Y$ be a stratified topological pseudomanifold of dimension $n$ with
stratification $\X=\{S\}$.  We will assume that all
strata of $Y$ are connected which implies that the frontier condition
holds: the closure of any stratum is a union of strata.
%
%
We also assume that there are finitely many strata.  The strata of $Y$
are partially ordered by $S\preccurlyeq T$ if and only if $S \subseteq
\overline T$.

We now briefly recall the definition of the category of perverse sheaves on
$Y$ following \cite{refnBeilinsonBernsteinDeligne}*{\S2.1}.

Fix a $\ZZ$-valued function $p$ on the set of strata, the
\emph{perversity} (in the sense of \cite{refnBeilinsonBernsteinDeligne});
we assume the perversity satisfies $p(S)\ge p(T)$ when $S\preccurlyeq T$.
Our main interest is when $p$ is one of the two \emph{middle
perversities}:
\begin{equation}
  \label{eqnMiddlePerversities}
  p_-(S) = -\left\lfloor\frac{\dim S+1}{2}\right\rfloor\ , \qquad
  p_+(S) = -\left\lfloor\frac{\dim S}{2}\right\rfloor\ .
\end{equation}
If $Y$ has only even dimensional strata, both $p_-$ and $p_+$ are
equal to the self-dual perversity $p_{1/2}(S)=-(1/2)\dim S$.  (A
non-middle perversity is introduced in
\S\ref{ssectLinkCohomologyGeneralDepth} for technical reasons.)

Let $\Derc^b(Y)$ denote the bounded derived category of $\CC$\n-sheaves on
$Y$, constructible with respect to $\X$.  (For an algebraic variety one
usually does not fix a stratification, but since our main focus here is
$Y=\Xhat$ it seems appropriate.)  An object in $\Derc^b(Y)$ may be
represented by a complex of sheaves $\Sh$ on $Y$ whose local cohomology
sheaves $H(i_S^*\Sh)$ along a stratum $S$ are locally constant and finitely
generated; such a locally constant sheaf may be equivalently viewed as a
$\pi_1(S)$\n-module, once we have picked as base point in $S$ (omitted in
the notation).  Since $Y$ is a stratified pseudomanifold, it can be shown
that $H(i_S^!\Sh)$ is likewise locally constant and finitely generated
\cite{refnBorelIntersectionCohomology}*{V,\S3}.

Define the full subcategory $\p
\Derc^{\le0}(Y)$ of $\Derc^b(Y)$ to consist of objects $\Sh$ which satisfy
\begin{equation}\label{eqnPerverseVanishing}
  \text{$H^k(\i_S^*\Sh) = 0 $ for $k> p(S)$, for all strata $S$}\ ;
\end{equation}
Likewise define $\p \Derc^{\ge0}(Y)$ to consist of objects $\Sh$ which satisfy
\begin{equation}\label{eqnPerverseCovanishing}
  \text{$H^k(\i_S^!\Sh) = 0 $ for $k< p(S)$, for all strata $S$}\ .
\end{equation}
We refer to \eqref{eqnPerverseVanishing} and \eqref{eqnPerverseCovanishing}
as the \emph{perverse vanishing} condition and the \emph{perverse
  covanishing} condition respectively.  As usual, we let $\p \Derc^{\le
  k}(Y) = \p \Derc^{\le 0}(Y)[-k]$ and $\p \Derc^{\ge k}(Y) = \p \Derc^{\ge
  0}(Y)[-k]$.

The pair $(\p \Derc^{\le0}(Y), \p \Derc^{\ge0}(Y) )$ forms a \t-structure;
its heart, $\p\Derc^{\le0}(Y) \cap \p\Derc^{\ge0}(Y)$, is the category of
\emph{$p$\n-perverse sheaves} on $Y$, denoted $\Perv(Y)=\Perv_p(Y)$.  For
$p=p_-$ (resp.\ $p_+$) we simply write $\Perv_-(Y)$ (resp.\ $\Perv_+(Y)$).

For $\Sh\in \Derc^b(Y)$, $D_S(\i_S^!\Sh) = \i_S^*D_Y(\Sh)$ where $D_Y$ and
$D_S$ denote the respective Verdier duality involutions.  Set
\[
p^*(S)=-p(S)-\dim S \ ,
\]
the \emph{dual} perversity.  Then $D_Y$ sends $p$\n-perverse sheaves to
$p^*$\n-perverse sheaves.  From \eqref{eqnMiddlePerversities} we see that
$p_-(S) + p_+(S) = -\dim S$ and thus $p_-^*=p_+$ and $p_+^*=p_-$.  In
particular, $D_Y$ interchanges $\Perv_-(Y)$ and $\Perv_+(Y)$.

\begin{warning}
In \S\ref{ssectReductiveConstructible}, after introducing the reductive
Borel-Serre compactification, we will impose a condition we call
\emph{reductively constructible} on our objects which is more appropriate
to use on $\Xhat$.  All of the material of the current section as well as
\S\ref{sectExtensions} hold without change under the assumption of
reductively constructible (with a minor exception noted later).
\end{warning}

\subsection{Simple perverse sheaves}
For every stratum $T$ of $Y$ and every irreducible local system $\EE$ on
$T$ (corresponding to an irreducible representation $E$ of $\pi_1(T)$)
there is a simple object
\begin{equation}\label{eqnSimpleObject}
\P_T(\EE) = \P_{p,T}(\EE) =  \tau_{\le p(S_1)-1}j_{S_1*} \dots
\tau_{\le p(S_N)-1}j_{S_N*}(i_{T*}\EE[-p(T)])
\end{equation}
of $\Perv(Y)$ and these are all the simple objects.  Here $S_1,\dots,S_N$
is an enumeration of the strata $S\prec T$ such that if $S_i \prec S_j$
then $i<j$.

$\P_T(\EE)$ is supported on $\overline T$ and satisfies the
usual perverse vanishing and covanishing condition on $T$:
\[
\text{$H^k(\i_T^*\P_T(\EE)) = 0$ for $k>p(T)$ and $H^k(\i_T^!\P_T(\EE))
  = 0 $ for $k< p(T)$}\ .
\]
In fact, $\i_T^*\P_T(\EE) = \i_T^!\P_T(\EE) = \EE[-p(T)]$.  However
$\P_T(\EE)$ satisfies stronger conditions on strata smaller than $T$:
\begin{equation}\label{eqnICVanishing}
  \text{$H^k(\i_S^*\P_T(\EE)) = 0 $ for $k\ge p(S)$, for all strata $S\prec T$}
\end{equation}
and
\begin{equation}\label{eqnICCovanishing}
  \text{$H^k(\i_S^!\P_T(\EE)) = 0 $ for $k\le p(S)$, for all strata
    $S\prec T$}\ .
\end{equation}
We call \eqref{eqnICVanishing} and  \eqref{eqnICCovanishing} the
\emph{intersection cohomology vanishing} condition and the 
\emph{intersection cohomology  covanishing} condition respectively.

Let $S$ be a stratum of $Y$, $V_S$ an open union of strata containing $S$
as a minimal stratum, and $j_S\colon U_S = V_S\setminus S \hookrightarrow
V_S$.  Define the \emph{link cohomology functor} $\lk_S\colon \Derc^b(U_S)
\to \Derc^b(S)$ by
\[
\lk_S A =  i_S^* j_{S*} A\ .
\]
The local cohomology of $\lk_S A$ is the cohomology of $A$ pulled back to
the topological link of the stratum $S$.  If $A$ is defined on a subset
larger than $U_S$, for example $Y$, we write $\lk_S A$ for $\lk_S(A|_{U_S})$.  Thus
it fits into a distinguished triangle for $A \in  \Derc^b(Y)$:
\[
i_S^!A \to i_S^*A \to \lk_S A \xrightarrow{[1]}\ .
\]
Now the values of the cohomology groups appearing in
\eqref{eqnICVanishing} and \eqref{eqnICCovanishing} in the other degrees
can be easily calculated:
\begin{equation}\label{eqnLocalCohomology}
  \text{$H^k(\i_S^*\P_T(\EE)) \cong H^k(\lk_S \P_T(\EE))$ for $k< p(S)$,
    for all strata $S\prec T$}
\end{equation}
and
\begin{equation}\label{eqnLocalCohomologyWithSupports}
  \text{$H^k(\i_S^!\P_T(\EE)) =   H^{k-1}(\lk_S \P_T(\EE))$
    for $k > p(S)$, for all strata $S\prec T$}\ .
\end{equation}
All degrees of the $\P_T(\EE)$-cohomology of the link of $S$ appear, split
between $\i_S^*\P_T(\EE)$ and (with a shift) $\i_S^!\P_T(\EE)$.

\subsection{Comparison with intersection cohomology}\label{ssectComparison}
Assume $p$ is one of the middle perversities $p_-$ or $p_+$.  Define
\begin{equation}\label{eqnShiftedPerversity}
\pbar_T(S) = p(S) - p(T) - 1, \qquad (S\preccurlyeq T)\ .
\end{equation}
Then by pulling out the shift in
\eqref{eqnSimpleObject} we see that
\[
\P_{p,T}(\EE) = \I_{\pbar_T}\C(\overline T,\EE)[-p(T)]
\]
where for $\pbar\colon \X(\overline T) \to \ZZ$ we set
\begin{equation}\label{eqnDelignesSheaf}
 \I_{\pbar}\C(\overline T,\EE) =  \tau_{\le \pbar(S_1)}j_{S_1*} \dots
\tau_{\le \pbar(S_N)}j_{S_N*}i_{T*}\EE \ .
\end{equation}
When $\pbar(S)$ depends only on $k=\codim_T S$, as is the case for
$\pbar_T(S)$ from \eqref{eqnShiftedPerversity}, the object $\I_{\pbar}\C(\overline T,\EE)$ is essentially Deligne's sheaf for
intersection cohomology as in \cite{refnBorelIntersectionCohomology}*{V}.

In fact $\pbar_T(S)$ agrees with one of the two ``classical'' middle
perversities of Goresky and MacPherson \cite{refnGoreskyMacPhersonIHTwo}:
\begin{equation*}
  m(k) = \left\lfloor\frac{k-2}{2}\right\rfloor, \qquad
  n(k) = \left\lfloor\frac{k-1}{2}\right\rfloor.
\end{equation*}
Specifically for $p=p_-$,
\begin{equation*}
  \pbar_T(S) = \begin{cases}
    n(\codim_T S) & \text{if $\dim T$ is odd,} \\
    m(\codim_T S) & \text{if $\dim T$ is even,}
  \end{cases}
\end{equation*}
and for $p=p_+$,
\begin{equation*}
  \pbar_T(S) = \begin{cases}
    m(\codim_T S) & \text{if $\dim T$ is odd,} \\
    n(\codim_T S) & \text{if $\dim T$ is even,}
  \end{cases}
\end{equation*}

Thus $\P_{p,T}(\EE) = \I_{\pbar_T}\C(\overline T,\EE)[-p(T)]$ is either
$\I_m\C(\overline T,\EE)[-p(T)]$ or $\I_n\C(\overline T,\EE)[-p(T)]$
depending on the parity of $\dim T$.

\subsection{A non-trivial extension of perverse sheaves}
Let $p$ be a middle perversity and let $\EE$ be an irreducible local system
on a stratum $T$ of $Y$.   While  $\P_{p,T}(\EE) = \I_{\pbar_T}\C(\overline
T,\EE)[-p(T)]$ is a simple perverse sheaf in $\Perv_p(Y)$, in fact the object
\[
\widetilde \P_{p,T}(\EE) \equiv  \P_{p^*,T}(\EE)[p^*(T)-p(T)] =
\I_{\overline{p^*}_T}\C(\overline T,\EE)[-p(T)]
\]
is also a perverse sheaf in $\Perv_p(Y)$, though not necessarily simple.

To see this, note that the $p$\n-perversity condition on the stratum $T$ is
satisfied due to the shift.  To deal with the other strata, note first that
for a middle perversity
\[
p^*(S) - p(S) \in
\begin{cases}
  \{0,1\} & \text{if $p=p_-$,} \\
  \{0,-1\} & \text{if $p=p_+$,}
\end{cases}
\]
for any stratum $S$; here the value $0$ occurs if and only if $\dim S$ is
even.  Consequently
\begin{equation}
  \label{eqnDiscrepancyDifference}
  -1 \le   (p^*(S)-p(S)) - (p^*(T)-p(T)) \le 1
\end{equation}
for all strata $S$ and $T$.  Now for $S\prec T$, the intersection
cohomology vanishing condition \eqref{eqnICVanishing} with respect to $p^*$
implies that
\[
\text{$H^k(i_S^* \P_{p^*,T}(\EE)[p^*(T)-p(T)]) = 0$ for $k \ge
  p^*(S) - (p^*(T)-p(T)) $.}
\]
This implies the $p$\n-perverse vanishing condition
\eqref{eqnPerverseVanishing} for $k>p(S)$ since
\[
k \ge p(S)+1 \ge p(S) + (p^*(S)-p(S)) - ( p^*(T) -p(T)) =  p^*(S) -
(p^*(T)-p(T))
\]
by \eqref{eqnDiscrepancyDifference}.  The $p$\n-perverse covanishing
condition \eqref{eqnPerverseCovanishing} follows similarly.

The natural morphism $\I_m\C(\overline T,\EE)[-p(T)] \to
\I_n\C(\overline T,\EE)[-p(T)]$ thus induces short exact sequences
\[
0\to \P_{p,T}(\EE) \to \widetilde \P_{p,T}(\EE) \to \P \to 0 \qquad
\text{(if $\pbar_T=m$)}
\]
or
\[
0 \to \P \to \widetilde \P_{p,T}(\EE) \to \P_{p,T}(\EE) \to 0  \qquad
\text{(if $\pbar_T=n$)}
\]
in $\Perv_p(Y)$ and hence extensions in $\Perv_p(Y)$.

If $\P_{p,T}(\EE) \neq \widetilde \P_{p,T}(\EE)$, the resulting extensions
are not trivial.  For if $\widetilde\P_{p,T}(\EE) = \P \oplus
\P_{p,T}(\EE)$ in $\Perv_p(Y)$, then applying $i_S^*$ or $i_S^!$ and
shifting by $p(T)-p^*(T)$, we see that $\P[p(T)-p^*(T)]$ and
$\widetilde\P_{p^*,T}(\EE) = \P_{p,T}(\EE)[p(T)-p^*(T)]$ satisfy the
$p^*$\n-perverse conditions (since $\P_{p^*,T}(\EE)= \widetilde\P_{p,T}(\EE)
[p(T)-p^*(T)]$ does).  This implies that $\P_{p^*,T}(\EE)= \P[p(T)-p^*(T)]
\oplus \widetilde\P_{p^*,T}(\EE)$ in $\Perv_{p^*}(Y)$, contradicting the
fact that $\P_{p^*,T}(\EE)$ is simple in $\Perv_{p^*}(Y)$.

\begin{rems*}
  \hangindent\myleftmargin
  \textup{(i)} The inequality $\P_{p,T}(\EE) \neq
  \widetilde\P_{p,T}(\EE)$ will occur precisely when there is a odd
  codimension stratum $S \prec T$ with nonvanishing middle degree link
  cohomology
  \[
  H^k(\lk_S\P_{p,T}(\EE))\neq 0 \qquad \text{($k=(\codim_T S - 1)/2 +
    p(T) $).}
  \]
  \begin{enumerate}\setcounter{enumi}{1}
  \item
  If $Y=\Xhat$, the reductive Borel-Serre compactification of a locally
  symmetric variety, and $\EE$ is induced from an algebraic representation
  of $L_T$, then decomposition theorem for $\pi\colon \Xhat \to X^*$ (see
  \S\ref{ssectRBS}) implies that this extension always becomes trivial when
  pushed down to the Baily-Borel compactification.  In fact
  $\pi_*\P_{p,X}(\EE) = \pi_* \widetilde \P_{p,X}(\EE)$ by the solution to
  Rapoport's conjecture \cite{refnSaperLModules} since it applies to both
  middle perversities and they agree on $X^*$; more generally the
  theorem shows that $\pi_*\P_{p,X_P}(\EE)$ and  $\pi_* \widetilde
  \P_{p,X_P}(\EE)$ differ only in the coefficient system, but in the
  reductively constructible setting, the category of coefficient systems on
  a stratum is semisimple.
\item
When $Y=\Xhat$ (with $X$ not necessarily Hermitian), the inequality
$\P_{p,T}(\EE) \neq \widetilde\P_{p,T}(\EE)$ has an intimate relation with
the presence of infinite dimensional local $L^2$\n-cohomology on $\Xhat$.
  \end{enumerate}
\end{rems*}

\section{Extensions}\label{sectExtensions}

Fix a perversity $p$ (not necessarily middle).  In this section we indicate
how one may compute the extensions
$\Ext^1_{\Perv(Y)}(\P_T(\EE),\P_{T'}(\EE'))$ between two simple perverse
sheaves corresponding to strata $T$, $T'\in \X(Y)$.  Note that by the
theory of \t-structures,
\[
\Ext^1_{\Perv(Y)}(\P_T(\EE),\P_{T'}(\EE')) =
\Mor_{\Derc^b(Y)}(\P_T(\EE),\P_{T'}(\EE')[1]) \ .
\]

\begin{lem}\label{lemInductiveStep}
  Let $V$ be an open union of strata with $S$ a minimal stratum of $V$; set
  $U=V\setminus S$. For $A$, $B \in \Derc^b(V)$ and $k\in \ZZ$, assume
  $\Mor_{\Derc^b(U)}(j_S^* A,j_S^*B[k-1]) = \Mor_{\Derc^b(U)}(j_S^* A,j_S^*B[k]) =
  0$.  Then $\Mor_{\Derc^b(V)}(A,B[k]) \cong
  \Mor_{\Derc^b(S)}(i_S^*A,i_S^!B[k])$.
\end{lem}

\begin{proof}
The distinguished triangle $i_{S*}i_S^!B\to B \to j_{S*}j_S^* B
\xrightarrow{[1]}$ yields a long exact sequence
\begin{multline*}
\Mor_{\Derc^b(V)}(A,j_{S*}j_S^*B[k-1]) \longrightarrow
\Mor_{\Derc^b(V)}(A,i_{S*}i_S^!B[k])  \longrightarrow \\
\Mor_{\Derc^b(V)}(A,B[k])  \longrightarrow  \Mor_{\Derc^b(V)}(A,j_{S*}j_S^*B[k]) \ .
\end{multline*}
The outer two groups are zero by hypothesis, thus
\[
\Mor_{\Derc^b(V)}(A,B[k]) \cong \Mor_{\Derc^b(V)}(A,i_{S*}i_S^!B[k]) \cong
\Mor_{\Derc^b(S)}(i_S^*A,i_S^!B[k]) \ . \qedhere
\]
\end{proof}

\begin{lem}
  If all maximal strata of $\overline T \cap \overline {T'}$ are strictly
  smaller than $T$ and $T'$, then
  $\Mor_{\Derc^b(Y)}(\P_T(\EE),\P_{T'}(\EE')[k])=0$ for $k\le 1$.
\end{lem}

\begin{proof}
  Let $V$ be an open union of strata containing $Y \setminus \overline T
  \cap \overline {T'}$.  We will prove that
\[
\Mor_{\Derc^b(V)}(\P_T(\EE),\P_{T'}(\EE')[k])=0, \qquad k\le1,
\]
  by induction on the number of strata in $V$.  When $V=Y$ this proves the
  lemma.

  If $V = Y \setminus \overline T \cap \overline {T'}$, the supports of
  $\P_T(\EE)$ and $\P_{T'}(\EE')$ are disjoint and
\[
  \Mor_{\Derc^b(V)}(\P_T(\EE),\P_{T'}(\EE')[k])=0, \qquad \text{for all
    $k$.}
\]

  For larger $V$, let $S$ be a minimal stratum in $V\cap(\overline T \cap
  \overline {T'})$ and set $U = V \setminus S$.  By Lemma ~
  \ref{lemInductiveStep} (which applies by the inductive hypothesis)
  \[
  \Mor_{\Derc^b(V)}(\P_T(\EE),\P_{T'}(\EE')[k]) \cong
  \Mor_{\Derc^b(S)}(i_S^*\P_T(\EE),i_S^!\P_{T'}(\EE')[k]) \ .
  \]
  for $k\le 1$.  Since $S$ is dominated by a maximal stratum in $\overline
  T \cap \overline {T'}$ which is thus strictly smaller than $T$ and $T'$,
  the intersection cohomology vanishing and covanishing conditions apply to
  yield $i_S^*\P_T(\EE) \in \Derc^{\le p(S)-1}(S)$ and
  $i_S^!\P_{T'}(\EE')[k] \in \Derc^{\ge p(S)+1-k}(S)$.  Now $p(S)-1 <
  p(S)+1-k$ for $k\le 1$, so
  $\Mor_{\Derc^b(S)}(i_S^*\P_T(\EE),i_S^!\P_{T'}(\EE')[k]) =0$ by
  \eqref{eqnTstructureOne}.
\end{proof}

We conclude that a non-trivial extension can only exist when $T=T'$, $T
\prec T'$, or $T\succ T'$.

\begin{lem}\label{lemStartOfExtension}
Let $V_T = (Y\setminus \overline{T}) \cup T$ and similarly for $T'$.
\begin{enumerate}
  \item\label{itemTEqualRprime} If $T = T'$,
    $\Mor_{\Derc^b(V_T)}(\P_T(\EE),\P_{T}(\EE')[1]) \cong
     \Mor_{\Derc^b(T)}(\EE,\EE'[1])$.
\item\label{itemTLessThanTprime}  If $T \prec T'$,
$\Mor_{\Derc^b(V_T)}(\P_T(\EE),\P_{T'}(\EE')[1]) \cong
  \Mor_{\pi_1(T)}(E,H^{p(T)}(\lk_T \P_{T'}(\EE')))$.
\item\label{itemTGreaterThanTprime} If $T \succ  T'$,
  $\Mor_{\Derc^b(V_{T'})}(\P_T(\EE),\P_{T'}(\EE')[1]) \cong 
  \Mor_{\pi_1(T')}(H^{p(T')-1}(\lk_{T'} \P_T(\EE)),E')$.
\end{enumerate}
\end{lem}

\begin{proof}
 The case $T=T'$ follows from Lemma ~\ref{lemInductiveStep}.  For $T\prec
 T'$ we calculate
  \begin{align*}
    \Mor_{\Derc^b(V_T)}(\P_T(\EE),\P_{T'}(\EE')[1]) &\cong
    \Mor_{\Derc^b(T)}(\EE[-p(T)],i_T^!\P_{T'}(\EE')[1]) && \text{(by Lemma
      ~\ref{lemInductiveStep})} \\
    &\cong \Mor_{\pi_1(T)}(E, H^{p(T)+1}(i_T^!\P_{T'}(\EE'))) &&
    \text{(by \eqref{eqnICCovanishing} and
      Lemma~\ref{lemMorphismsWithOneOverlap})} \\
    &\cong \Mor_{\pi_1(T)}(E, H^{p(T)}(\lk_T \P_{T'}(\EE'))) &&
    \text{by \eqref{eqnLocalCohomologyWithSupports}}
  \end{align*}
  The case where $T\succ T'$ is similar.
\end{proof}

\begin{lem}\label{lemProlongExtension}
  Assume $S\prec T$ and $S\prec T'$.  Let $V$ be an open union of strata
  with $S\subset V$ a minimal stratum and set $U= V\setminus S$.  A
  morphism $\phi_U\in \Mor_{\Derc^b(U)}(j_S^*\P_T(\EE),j_S^*\P_{T'}(\EE')[1])$
  prolongs to a morphism $\phi_V$ over $V$ if and only if
      \begin{equation}\label{eqnTLessThanTprimeExtensionCondition}
      H^{p(S)-1}(\lk_S(\phi_U))=0 \text{ in } \Mor_{\pi_1(S)}(H^{p(S)-1}(\lk_S
      \P_T(\EE)), H^{p(S)}(\lk_S \P_{T'}(\EE'))) \ .
      \end{equation}
If this prolongation exists it is unique.
\end{lem}

\begin{proof}
  We apply $\Mor_{\Derc^b(V)}(\P_T(\EE),\cdot)$ to the distinguished
  triangle $i_{S*}i_S^!\P_{T'}(\EE') \to \P_{T'}(\EE') \to i_{S*} j_S^*
  \P_{T'}(\EE') \xrightarrow{[1]}$ which yields, after using adjointness,
  the long exact sequence
\begin{multline*}
\Mor_{\Derc^b(S)}(i_S^*\P_T(\EE),i_S^!\P_{T'}(\EE')[1])  \longrightarrow
\Mor_{\Derc^b(V)}(\P_T(\EE),\P_{T'}(\EE')[1])  \longrightarrow \\
\Mor_{\Derc^b(U)}(j_S^*\P_T(\EE),j_S^*\P_{T'}(\EE')[1]) \longrightarrow 
\Mor_{\Derc^b(S)}(i_S^*\P_T(\EE),i_S^!\P_{T'}(\EE')[2])
\ .
\end{multline*}
Since $i_S^*\P_T(\EE) \in \Derc^{\le p(S)-1}(S)$ by \eqref{eqnICVanishing}
and $i_S^!\P_{T'}(\EE')[1] \in \Derc^{\ge p(S)}(S)$ by
\eqref{eqnICCovanishing}, the leftmost term is $0$ by
\eqref{eqnTstructureOne}.  This shows that $\phi_V$ is unique if it exists.
It exists if and only if $\phi_U$ maps to zero in the last term, which by a
similar calculation and using Lemma ~\ref{lemMorphismsWithOneOverlap} is
\[
\Mor_{\pi_1(S)}(H^{p(S)-1}(i_S^*\P_T(\EE)),H^{p(S)-1}(i_S^!\P_{T'}(\EE')[2]))\ .
\]
By \eqref{eqnLocalCohomology} and \eqref{eqnLocalCohomologyWithSupports} this
is the group of morphisms in \eqref{eqnTLessThanTprimeExtensionCondition}.
\end{proof}

Note that the map on link cohomology at $S$ which needs to be zero in order
to prolong an extension is the map from the highest degree
of link cohomology occurring in $i_S^* \P_T(\EE)$ to the lowest degree of
link cohomology that does \emph{not} occur in $i_S^* \P_{T'}(\EE')$.

This section is summarized in the following
\begin{prop}
$\Ext^1_{\Perv(Y)}(\P_T(\EE),\P_{T'}(\EE'))$ is zero unless $T=T'$, $T \prec T'$
or $T\succ T'$.  
\begin{enumerate}
  \item For $T=T'$, $\Ext^1_{\Perv(Y)}(\P_T(\EE),\P_{T}(\EE'))$ is
  isomorphic to a subgroup of
  \[
  \Mor_{\Derc^b(V_T)}(\P_T(\EE),\P_{T}(\EE')[1]) \cong
  \Mor_{\Derc^b(T)}(\EE,\EE'[1]) \ .
\]
\item For $T \prec T'$, $\Ext^1_{\Perv(Y)}(\P_T(\EE),\P_{T'}(\EE'))$ is
  isomorphic to a subgroup of
\[
  \Mor_{\Derc^b(V_T)}(\P_T(\EE),\P_{T'}(\EE')[1]) \cong
  \Mor_{\pi_1(T)}(E,H^{p(T)}(\lk_T \P_{T'}(\EE')))\ .
\]
\item For $T \succ T'$, $\Ext^1_{\Perv(Y)}(\P_T(\EE),\P_{T'}(\EE'))$ is
  isomorphic to a subgroup of
\[
 \Mor_{\Derc^b(V_{T'})}(\P_T(\EE),\P_{T'}(\EE')[1]) \cong 
  \Mor_{\pi_1(T')}(H^{p(T')-1}(\lk_{T'} \P_T(\EE)),E')\ .
\]
\end{enumerate}
In either case, the group of extensions consists of those morphisms $\phi$ that
recursively satisfy 
\[
     H^{p(S)-1}(\lk_S(\phi|_{U_S}))=0 \text{ in } \Mor_{\pi_1(S)}(H^{p(S)-1}(\lk_S
      \P_T(\EE)), H^{p(S)}(\lk_S \P_{T'}(\EE'))) \ .
\]
for all $S$ satisfying $S\prec T$ and $S\prec T'$.  Here $V_S$ is an open
union of strata containing $S$ as a minimal stratum and $U_S = V_S\setminus
S$; the condition above allows one to uniquely extend $\phi$ from $U_S$ to
$V_S$.
\end{prop}

\begin{rems*}
    \hangindent\myleftmargin
  \textup{(i)} 
The condition \eqref{eqnTLessThanTprimeExtensionCondition} needed to
prolong an extension can be nontrivial even in the case $T=T'$; one
example of an extension of coefficient systems failing this condition (and
hence not prolonging to an extension of perverse sheaves) is
\cite{refndeCataldoMiglioriniDecompositionTheorem}*{Examples 2.2.5 and
  2.7.1}.  The point is that the intermediate extension function $\EE
\mapsto \P_T(\EE)$ \eqref{eqnSimpleObject} is not exact even though it
preserves injective and surjective maps
\cite{refndeCataldoMiglioriniDecompositionTheorem}*{\S2.7}.
  \begin{enumerate}\setcounter{enumi}{1}
  \item In \S\ref{ssectReductiveConstructible} we will define a more
    refined notion of constructibility in which local systems on a stratum
    $T$ are associated to algebraic representations of a reductive group
    $L_T$.  All the results of this section hold in this context except one
    should replace $\Mor_{\pi_1(T)}$ by $\Mor_{L_T}$ as needed.  In
    addition, since algebraic representations of a reductive group are
    semisimple, there are no extensions between simple objects when $T=T'$.
  \end{enumerate}
\end{rems*}

\section{Reductive Borel-Serre compactification}\label{sectRBS}
The reductive Borel-Serre compactification $\Xhat$ first appeared in work
of Zucker \cite{refnZuckerWarped} and has grown in importance far beyond
its original use.  We recall it following
\cite{refnGoreskyHarderMacPherson}, focusing mainly on its structure rather
than its construction.

\subsection{Stratification}
For $G$ a reductive algebraic group defined over $\QQ$ and $\Gamma$ an
arithmetic subgroup, let $X = X_G = \Gamma \backslash G(\RR) / A_G(\RR)K$
be the corresponding locally symmetric space.  Here $A_G$ is the maximal
$\QQ$\n-split torus in the center of $G$, and $K$ is a maximal compact
subgroup of $G(\RR)$.

For a parabolic $\QQ$\n-subgroup $S$ of $G$, let $L_S = S/N_S$ denote its
reductive Levi quotient, where $N_S$ denotes the unipotent radical of $S$;
it is also defined over $\QQ$.  We have an almost direct product $L_S = M_S
A_S$ where $M_S = \lsp0 L_S = \bigcap_\chi \ker \chi^2$, where $\chi$ runs over
characters of $L_S$ defined over $\QQ$.  Let $\Gamma_{N_S} = \Gamma\cap
N_S$ and $\Gamma_{L_S} = (\Gamma\cap S) / (\Gamma\cap N_S)$ be the induced
arithmetic subgroups.  Starting from $L_S$ and $\Gamma_{L_S}$ we again
obtain a locally symmetric space $X_S$.  If $S'$ is $\Gamma$\n-conjugate to
$S$, then $X_{S'}$ and $X_S$ may be canonically identified.

The reductive Borel-Serre compactification of $X$ is the stratified
pseudomanifold $\Xhat = \coprod_{S} X_S$ (where $S$ ranges over the
$\Gamma$\n-conjugacy classes of parabolic $\QQ$\n-subgroups of $G$) endowed
with an appropriate topology.  The closure of a stratum $X_S$ is denoted
$\Xhat_S$; it is indeed the reductive Borel-Serre compactification of $X_S$
and its strata $X_T$ correspond to $\Gamma$\n-conjugacy classes of
parabolic $\QQ$\n-subgroups having a representative $T\subseteq S$.

From now on we we will abuse notation by using the same letter $S$ to
denote a stratum $X_S$, the corresponding $\Gamma$-conjugacy class of
parabolic $\QQ$\n-subgroups, and a representative of that class.  Thus the
partial order $S \preccurlyeq T$ on strata corresponds on parabolic
subgroups to $S \subseteq \lsp{\gamma} T$ for some $\gamma \in \Gamma$.

The \emph{depth} of a stratum $S$ of $\Xhat$, $\depth S$, is the maximal
length $d$ of a chain $S = S_0 \prec S_1 \prec \dots \prec S_d = X$.  If
$S\preccurlyeq T$, the \emph{relative depth} of $S$ viewed as a stratum of
$\overline T$, $\depth_T S$, is the maximal length of a chain $S = S_0
\prec S_1 \prec \dots \prec S_d = T$.  We have $\depth_T S = \dim A_S -
\dim A_T$.

An algebraic representation $E$ of $L_S$ induces a representation of
$\Gamma_{L_S}$ and hence a local system $\EE$ on $S$.

\subsection{Reductive constructibility}\label{ssectReductiveConstructible}
A constructible complex of sheaves $\Sh$ on a pseudomanifold has the
property that for all strata $S$, the cohomology sheaves of $i_S^* \Sh$,
$i_S^! \Sh$, and $\lk_S \Sh$ are all finitely generated locally constant
and hence are associated to finite-dimensional representations of
$\Gamma_{L_S}$.  Furthermore the maps on cohomology induced by the
distinguished triangle $i_S^!\Sh \to i_S^*\Sh \to \lk_S \Sh
\xrightarrow{[1]}$ are morphisms of $\Gamma_{L_S}$\n-representations.  A
\emph{reductively constructible} complex of sheaves $\Sh$ on $\Xhat$ is as
above but has been enriched with extra structure so that all of these
locally constant sheaves arise from algebraic representations of $L_S$ and
that the morphisms above are morphisms of $L_S$\n-modules.  Morphisms
between such sheaves must also induce morphisms of $L_S$\n-modules on the
cohomology sheaves over a stratum.

Rather than construct the derived category of reductively constructible
sheaves $\Derived_{rc}^b(\Xhat)$ directly as suggested above, we note instead that if one starts
with the category of $\mathscr L$\n-modules constructed in
\cite{refnSaperLModules} and pass to the homotopy category, one obtains the
desired category of reductively constructible sheaves.  (In the homotopy
category of  $\mathscr L$\n-modules, every quasi-isomorphism is already an
isomorphism and thus one does not need to localize further.)  We will
discuss the details elsewhere.

Exactly as in \S\ref{sectPerverse} one can define a perverse \t-structure
on $\Derived_{rc}^b(\Xhat)$ and obtain a category of perverse sheaves
which we again denote $\Perv(\Xhat)$.  The description of simple objects
is the same except we start with $\EE$ on $T$ coming from an algebraic
representation of $L_T$.  The fact that the pushforward functors in the
definition of the simple perverse sheaves do preserve reductive
constructibility is a consequence of the theorems of Nomizu, van Est, and
Kostant which we recall in \S\ref{sectLinkCohomology}.  The results in
\S\ref{sectExtensions} all remain true with identical proofs however one
must replace $\pi_1(S)$\n-morphisms with $L_S$\n-morphisms in Lemmas
~\ref{lemStartOfExtension} and \ref{lemProlongExtension}.

Enriching our objects $\Sh$ so that $H(i_S^*\Sh)$ has an action of the
central split torus $A_S\subset L_S$ is new data.  However aside from this,
passing from the constructible category to the reductively constructible
category is not as major a change as it may appear at first glance due to
the Borel density theorem \cite{refnBorelDensity} and Margulis
superrigidity \cite{refnMargulisDiscreteSubgroups}.

\subsection{Links}
We recall the link of a stratum $S$ of $\Xhat$ following
\cite{refnGoreskyHarderMacPherson} (see also \cite{refnSaperCDM}*{\S7}).

For any parabolic $\QQ$\n-subgroup $S$, the Levi quotient $L_S$ acts (via a
lift to $S$) by conjugation on the Lie algebra $\nn_S$ of $N_S$.  Though
this action depends on the lift, the weights by which $A_S$ acts on $\nn_S$
are well-defined and have a unique basis
denoted $\Delta_S$.  If $S\subseteq S'$, then $N_{S'} \lhd N_S$ and group $A_{S'}$
may be naturally viewed as a subgroup of $A_S$.  We let
$\Delta_S^{S'}\subseteq \Delta_{S'}$ be those weights that restrict trivially to
$A_{S'}$.  The correspondence ${S'} \mapsto \Delta_S^{S'}$ is an order preserving
bijection between parabolic $\QQ$\n-subgroups containing $S$ and subsets of
$\Delta_S$.  Restriction to $A_{S'}$ yields a bijection between
$\Delta_S\setminus \Delta_S^{S'}$ and $\Delta_{S'}$.  Note that $\Delta_S^G =
\Delta_S$.

Let $|\Delta_S|$ be a topological simplex with vertices indexed by the
elements of $\Delta_S$.  Give $|\Delta_S|$ the stratification by its open
faces $|\Delta_S^{S'}|^\circ = \operatorname{int} |\Delta_S^{S'}|$ indexed
by ${S'}\supsetneq S$.  Note that
\[
\dim |\Delta_S| = \depth S \ .
\]

The topological link of the stratum $S\subset \Xhat$ is
\begin{equation}\label{eqnLink}
\Lk_S = \left. \left(\, \Gamma_{N_S}\backslash N_S(\RR) \times |\Delta_S|\,
\right) \right/ \sim
\end{equation}
where if $t\in |\Delta_S^{S'}|^\circ$, $(n,t) \sim (n'n,t)$ for all $n'\in
N_{S'}(\RR)$.  Thus the intersection of the link of $S$ with higher strata
${S'}\supset S$ is
\begin{equation}\label{eqnLinkIntersectStratum}
\Lk_S \cap {S'} = \Gamma_{N_S^{S'}} \backslash N_S^{S'}(\RR) \times
|\Delta_S^{S'}|^\circ\ ,
\end{equation}
where $N_S^{S'} = N_S/N_{S'}$ and $\Gamma_{N_S^{S'}} = \Gamma_{N_S}/\Gamma_{N_{S'}}$.

More generally, the topological link of $S$ viewed as a stratum of
$\overline T$ (for $S\preccurlyeq T$) is
\begin{equation}\label{eqnLinkGeneral}
\Lk_S^T = \left. \left(\, \Gamma_{N_S^T}\backslash N_S^T(\RR) \times
|\Delta_S^T|\, \right) \right/ \sim 
\end{equation}
and
\[
\dim |\Delta_S^T| = \depth_T S \ .
\]
If $S\preccurlyeq {S'} \preccurlyeq T$, we have again
\begin{equation}\label{eqnLinkIntersectStratumGeneral}
\Lk_S^T \cap {S'} = \Lk_S \cap {S'}  \Gamma_{N_S^{S'}} \backslash
N_S^{S'}(\RR) \times |\Delta_S^{S'}|^\circ \ .
\end{equation}

\section{Link cohomology in $\Xhat$}\label{sectLinkCohomology}
To actually calculate extensions we need to understand the link cohomology
functor $\lk_S = i_S^* j_{S*}$ on a stratum $S$.  We calculate the link
cohomology of $\P_T(\EE)$ for strata $S$ of relative depth $1$, $2$, and $3$,
following the formula given in \cite{refnSaperLModules}*{\S5.5} (see also
\cite{refnSaperCDM}*{\S18.4}).

\subsection{Kostant's theorem}
The calculation of $H(\lk_S \P_T(\EE))$ given later will involve the Lie
algebra cohomology $H(\nn_S^T,E)$ where $\nn_S^T$ is the Lie algebra of
$N_S^T$.  The adjoint action of $L_S$ on $\nn_S^T$ (via choice of a lift)
induces an defined action of $L_S$ on $H(\nn_S^T,E)$ which is independent
of the lift.  We begin by recalling Kostant's theorem \cite{refnKostant}
which gives a decomposition of $H(\nn_S^T,E)$ as an $L_S$\n-module.

We first consider $T=G$ and let $E$ be an irreducible algebraic
representation of $G$.  Choose a Cartan subalgebra for the Lie algebra
$\ll_S$ of $L_S$; it lifts to a Cartan subalgebra of the Lie algebra $\gg$
of $G$.  Choose an order on the roots of $\gg$ so that the roots in $\nn_S$
are all positive.  Let $\rho$ be one-half the sum of the positive roots of
$\gg$.  Let $W=W^G$ be the Weyl group of the root system of $\gg$ with
corresponding length function $\l(w)$.  Let $W^S\subseteq W$ be the
subgroup generated by the simple reflections in roots of $\ll_S$.  Let $W_S
\subseteq W$ be the set of unique minimal length representatives of cosets
in $W^S\backslash W$. Thus there is a product decomposition $W = W^SW_S$.
For $E$ an irreducible algebraic representation of $G$ with highest weight
$\lambda$, Kostant's theorem says
\[
 H(\nn_S, E) = \bigoplus_{w\in W_s} H^{\l(w)}(\nn_S,E)_w[-\l(w)]
\]
where $H^{\l(w)}(\nn_S,E)_w$ is an irreducible $L_S$\n-module and has
highest weight $w(\lambda+\rho)-\rho$.

Consider now $H(\nn_S^T, E)$ where $S\prec T$ and $E$ is an
irreducible algebraic representation of $L_T$ with highest weight
$\lambda$.  We replace $G$ above by $L_T$.  Thus let $\rho^T$ be one-half
the sum of the positive roots of $\ll_T$, let  $W^T$ be the Weyl group of
the root system of $\ll_T$.  We have  a decomposition $W^T = W^S W_S^T$,
where $W_S^T$ consists of the minimal length representatives of the cosets
in $W^S\backslash W^T$.  Kostant's theorem here says
\[
 H(\nn_S^T, E) = \bigoplus_{w\in W_s^T} H^{\l(w)}(\nn_S^T,E)_w[-\l(w)]
\]
where $H^{\l(w)}(\nn_S^T,E)_w$ has highest weight
$w(\lambda+\rho^T)-\rho^T$.

\subsection{Relative depth $1$}
Let $\EE$ be a local system on $T$ corresponding to an algebraic
representation $E$ of $L_T$.  By \eqref{eqnLinkIntersectStratumGeneral},
since $\depth_T S = 1$, the link $\Lk^T_S$ is simply the compact
nilmanifold $\Gamma_{N_S^T}\backslash N_S^T(\RR)$.  Thus
\begin{equation}\label{eqnOpenLinkCohomology}
H(\lk_S \P_T(\EE)) = H(\Gamma_{N_S^T}\backslash
N_S^T(\RR);\EE)[-p(T)] \cong  H(\nn_S^T, E)[-p(T)]\ .
\end{equation}
Here we use the theorem of Nomizu and van Est \citelist{\cite{refnvanEst}
  \cite{refnNomizu}} for the last isomorphism with Lie algebra cohomology.
In fact the action of $L_S$ on $\nn_S^T$ induces an action on $H(\nn_S^T,E)$ and
the isomorphism in \eqref{eqnOpenLinkCohomology} is an isomorphism of
$L_S$\n-modules.

If we combine this with Kostant's theorem we obtain
\begin{equation}\label{eqnLinkCohomologyDepthOne}
H(\lk_S \P_T(\EE)) \cong \bigoplus_{w\in W_S^T}
H^{\l(w)}(\nn_S^T,E)_w[-\l(w) - p(T)] \ .
\end{equation}

Note that to go from here to $H(i_S^* \P_T(\EE))$, according to
\eqref{eqnICVanishing} and \eqref{eqnLocalCohomology},  we need to truncate
this cohomology, leaving just  degrees $< p(S)$.  Thus the sum will now
only include $w$ satisfying  $\l(w) + p(T) \le p(S) -1$, or
\[
\l(w) \le p(S) - p(T) - 1 = \pbar_T(S) \ .
\]
Thus for relative depth $1$
\begin{equation}\label{eqnLocalCohomologyDepthOne}
H(\i_S^* \P_T(\EE)) \cong \bigoplus_{\substack{w\in W_S^T \\
                                     \l(w) \le \pbar_T(S)}}
H^{\l(w)}(\nn_S^T,E)_w[-\l(w) - p(T)]
\end{equation}
(compare \S\ref{ssectComparison}).

\subsection{Relative depth $2$}
If $\depth_T S = 2$, the simplex $|\Delta_S^T|$ is a $1$\n-simplex.  The
endpoints correspond to the two intermediate strata: $S \prec Q_1, Q_2
\prec T$.  The cohomology of the interior of the link is the same as
\eqref{eqnLinkCohomologyDepthOne}, however each term may be truncated at
$Q_1$ or at $Q_2$ or at both --- see \eqref{eqnICVanishing} as well as
\eqref{eqnLocalCohomologyDepthOne}.  The terms that remain are those that
are not truncated at either side or (with an additional shift by $-1$)
those that are truncated at both sides.  This corresponds to the fact that
the cohomology of a $1$\n-simplex relative to either one of its endpoints
vanishes, while the cohomology relative to both endpoints is one
dimensional in degree $1$.

To precisely express these side truncations, write $W^T = W^{Q_1} W_{Q_1}^T
= W^S W_S^{Q_i}W_{Q_i}^T$.  In fact  $W_S^T = W_S^{Q_i}W_{Q_1}^T$
and for $w\in W_S^T$ we decompose $w = w^{Q_i} w_{Q_i}$ accordingly.
Define
\[
\l_{Q_i}(w) = \l(w_{Q_i})\qquad (w = w^{Q_i} w_{Q_i})\ .
\]
Then we have the formula
\begin{multline}\label{eqnLinkCohomologyDepthTwo}
H(\lk_S \P_T(\EE)) \cong \bigoplus_{\substack{w\in W_S^T\\
                                     \l_{Q_1}(w) \le \pbar_T(Q_1)\\
                                     \l_{Q_2}(w) \le \pbar_T(Q_2)}}
H^{\l(w)}(\nn_S^T,E)_w[-\l(w) - p(T)]  \\
\oplus  \bigoplus_{\substack{w\in W_S^T\\
                                     \l_{Q_1}(w) > \pbar_T(Q_1)\\
                                     \l_{Q_2}(w) > \pbar_T(Q_2)}}
H^{\l(w)}(\nn_S^T,E)_w[-\l(w) - p(T)- 1]
\end{multline}

\subsection{Relative depth $3$ and
  higher}\label{ssectLinkCohomologyGeneralDepth}
One can similarly write an explicit formula for $H(\lk_S \P_T(\EE))$ when
$\depth_T S = 3$ however there are many more  possible configurations of
truncations on the $6$ intermediate strata $Q$; these are illustrated in
\cite{refnSaperCDM}*{\S18.4, Figure~24} and result in additional degree
shifts of $0$, $-1$, and $-2$.

There is a general formula for arbitrary depth that involves the
intersection cohomology of $|\Delta_S^T|$ associated to a special
perversity:
\[
\pbar_{T,w}(Q)= \pbar_T(Q) - \l_{Q}(w) = p(Q) - p(T) - 1- \l_Q(w)
\qquad (S\preccurlyeq Q \preccurlyeq T)\ .
\]
Let
\[
I_{\pbar_{T,w}}H^\cdot (|\Delta_S^T|) = H^\cdot(\I_{\pbar_{T,w}}\C(|\Delta_S^T|,\ZZ))
\]
be the hypercohomology of the corresponding Deligne sheaf
\eqref{eqnDelignesSheaf}.  Then
\begin{equation}
H(\lk_S \P_T(\EE)) \cong  \bigoplus_{w\in W_S^T}
H^{\l(w)}(\nn_S^T,E)_w[-\l(w)-p(T)]\otimes I_{\pbar_{T,w}}H(|\Delta_S^T|) \ .
\end{equation}

\section{Example Computations}\label{sectComputations}
To illustrate the results of \S\ref{sectExtensions}, we calculate
all extensions between simple perverse sheaves for the reductive Borel-Serre
compactification $\Xhat$ of the locally symmetric spaces associated with
$G=\Sp(4,\RR)$.  Since
$\Ext^1_{\Perv}(\P_T(\EE),\P_{T'}(\EE'))= 0$ unless $T\prec T'$ or
$T\succ T'$, we can assume both $T$ and $T'$ are standard parabolic
$\QQ$\n-subgroups.

For brevity we only consider $p=p_-$ and we give complete details for
$T\succ T'$; the results for $T\prec T'$ are summarized
in \S\ref{ssectResultsTprecTprime}.

\subsection{Preliminaries}
Here the root system is type $C_2$ with simple $\QQ$\n-roots
$\Delta=\{\al_1= e_1-e_2, \al_2=2e_2\}$. 
The standard parabolic
$\QQ$\n-subgroups $P_I$ correspond to subsets  $I\subseteq \{1,2\}$
($\Delta^{P_I} = \{\alpha_i\}_{i\in I} \subseteq \Delta$) and so
form a lattice
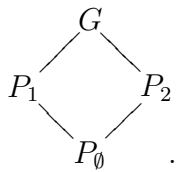
\begin{equation*}
\vcenter{\xymatrix @ur @M=1pt @R=1.5pc @C=1.5pc {
{P_1} \ar@{-}[r] \ar@{-}[d] & {G} \ar@{-}[d] \\
{P_\emptyset\rlap{\qquad .}} \ar@{-}[r] & {P_2}
}} \label{eqnParallelogram}
\end{equation*}
The real dimension of $X$ is $6$, while strata corresponding to $P_1$ and
$P_2$ have dimension $2$, and that corresponding to $P_\emptyset$ has
dimension $0$.  Thus the perversity values $p_-(S)$ are
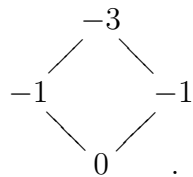
\begin{equation*}
\vcenter{\xymatrix @ur @M=3pt @R=1.5pc @C=1.5pc {
{-1} \ar@{-}[r] \ar@{-}[d] & {-3} \ar@{-}[d] \\
{0\rlap{\qquad .}} \ar@{-}[r] & {-1}
}} \label{eqnPerversityParallelogram}
\end{equation*}
To calculate link cohomology we need the Weyl group which is generated by
the reflections $s_1$ and $s_2$ in the simple roots modulo the relation
$(s_1s_2)^3=e$:
\[
W = W_{P_\emptyset} = \{e, s_1,s_2, s_1s_2, s_2s_1, s_1s_2s_1, s_2s_1s_2\}\ .
\]
For the intermediate strata we have
\[
W^{P_1}=W^{P_1}_{P_\emptyset} = \{e, s_1\} \qquad \text{and} \qquad
W^{P_2}=W^{P_2}_{P_\emptyset} = \{e, s_2\}
\]
and
\[
W_{P_1} = \{e, s_2, s_2s_1, s_2s_1s_2\} \qquad \text{and} \qquad
W_{P_2} = \{e, s_1, s_1s_2, s_1s_2s_1\} \ .
\]

\subsection{The case $T=G$}
We first set $T=G$ and $E$ an irreducible algebraic representation of $G$
with highest weight $\lambda$.  The associated classical perversity
$\pbar_G(S) = p(S)-p(G)-1$ has values
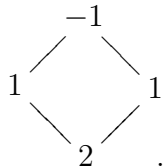
\begin{equation*}
\vcenter{\xymatrix @ur @M=3pt @R=1.5pc @C=1.5pc {
{1} \ar@{-}[r] \ar@{-}[d] & {-1} \ar@{-}[d] \\
{2\rlap{\qquad .}} \ar@{-}[r] & {1}
}}
\end{equation*}
We look for nonzero extensions in
$\Ext^1_{\Perv_-}(\P_{G}(\EE),\P_{T'}(\EE'))$ for $T'\prec G$, where $E'$
is an irreducible algebraic representation of $L_{T'}$.  To start
constructing such an extension, according to
Lemma~\ref{lemStartOfExtension}\ref{itemTGreaterThanTprime}, we need an
$L_{T'}$-module morphism
\[
H^{p(T')-1}(\lk_{T'} \P_G(\EE))  \longrightarrow E'\ .
\]

We calculate the above link cohomology.  We begin with $T'=P_\emptyset$ for
which we use
\eqref{eqnLinkCohomologyDepthTwo}.  In degree $p(T')-1$, a contribution
from the first direct sum of \eqref{eqnLinkCohomologyDepthTwo} requires
$\l(w) = \pbar_G(P_\emptyset)= 2$ and no truncation on either side.  This
means $w=s_1s_2$ or $s_2s_1$.  Now $\l_{P_1}(s_1s_2) = 1\le 1 =
\pbar_G(P_1)$ (since $s_1s_2 = (s_1)(s_2) \in W^{P_1}W_{P_1}$) and
$\l_{P_2}(s_1s_2) = 2 \not \le 1 = \pbar_G(P_2)$ (since $s_1s_2 =
(e)(s_1s_2) \in W^{P_1}W_{P_1}$).  Thus $s_1s_2$ cannot contribute to the
link cohomology since the corresponding class is truncated on one side but
not the other.  Similarly $s_2s_1$ does not contribute since
$\l_{P_1}(s_2s_1) = 2\not\le 1$ and $\l_{P_2}(s_2s_1) = 1 \le 1$.

A contribution in degree $p(T')-1$ from the second  direct sum of
\eqref{eqnLinkCohomologyDepthTwo} requires $\l(w)= \pbar_G(P_\emptyset)-1 =
1$ and  truncation on both sides.  Thus $w=s_1$ or $w_2$.   For
$w=s_1$, $\l_{P_1}(s_1) = 0 \not> \pbar_G(P_1)=1$ and $\l_{P_2}(s_1) = 1 \not>
\pbar_G(P_2)=1$ and thus does not contribute, and similarly for $w=s_2$.

Thus $H^{p(P_\emptyset)-1}(\lk_{P_\emptyset} \P_G(\EE)) =0$ and
\[
\Ext^1_{\Perv_-}(\P_{G}(\EE),\P_{P_\emptyset}(\EE'))=0
\]
for all $E'$.

We now consider the link cohomology at $T'=P_1$ or $P_2$.  Since we are
interested in degree $p(T')-1$, a contribution from
\eqref{eqnLinkCohomologyDepthOne} requires $\l(w) = \pbar_G(T') = 1$.  Thus
\begin{equation}
H^{p(T')-1}(\lk_{T'} \P_G(\EE)) =\begin{cases}
       H^1(\nn_{P_1},E)_{s_2}      & T'=P_1, \\
       H^1(\nn_{P_2},E)_{s_1}      & T'=P_2. \\
\end{cases}
\end{equation}

This means that $\Ext^1_{\Perv_-}(\P_{G}(\EE),\P_{P_1}(\EE'))$ can only be
nonzero (in which case it is $\CC$) when $E'$ has highest weight
$s_2(\lambda+\rho)-\rho$.   By Lemma ~\ref{lemProlongExtension}, this
extension will actually exist provided the map of $L_{P_1}$\n-modules
\[
\phi\colon H^{-2}(\lk_{P_1} \P_G(\EE))= H^1(\nn_{P_1},E)_{s_2} \longrightarrow E'
\]
(which is an isomorphism), induces the zero map
\begin{equation}\label{eqnMapOnLink} 
H^{p(P_\emptyset)-1}(\lk_{P_\emptyset} \P_G(\EE)) \longrightarrow
H^{p(P_\emptyset)}(\lk_{P_\emptyset} \P_{P_1}(\EE')).
\end{equation}
However we have already seen that
$H^{p(P_\emptyset)-1}(\lk_{P_\emptyset} \P_G(\EE)) =0$, thus this condition
is satisfied.

So we find that
\begin{equation}
\Ext^1_{\Perv_-}(\P_{G}(\EE),\P_{P_1}(\EE'))= \begin{cases}
        \CC   & \text{if $E'$ has highest weight $s_2(\lambda+\rho)-\rho$,} \\
        0     & \text{otherwise.}
\end{cases}
\end{equation}
A similar argument shows
\begin{equation}
\Ext^1_{\Perv_-}(\P_{G}(\EE),\P_{P_2}(\EE'))= \begin{cases}
        \CC   & \text{if $E'$ has highest weight $s_1(\lambda+\rho)-\rho$,} \\
        0     & \text{otherwise.}
\end{cases}
\end{equation}

\subsection{The case $T=P_1$ and $P_2$}
Fix $i=1$ or $2$, and say $E$ is an irreducible algebraic representation of
$L_{P_i}$ with highest weight $\lambda$.  By Lemma
~\ref{lemStartOfExtension}\ref{itemTGreaterThanTprime}
\[
\Ext^1_{\Perv_-}(\P_{P_i}(\EE),\P_{P_\emptyset}(\EE')) \cong 
\Mor_{L_{P_\emptyset}}(H^{p(P_\emptyset)-1}(\lk_{P_\emptyset} \P_{P_i}(\EE)), E')
\]
for any irreducible algebraic representation $E'$ of $L_{P_\emptyset}$.
Since $L_{P_\emptyset}$ is just the maximal torus, $E'= \CC_\chi$ where
$\chi $ is the character by which $L_{P_\emptyset}$ acts.  We
calculate by \eqref{eqnLinkCohomologyDepthOne} that
\[
H^{p(P_\emptyset)-1}(\lk_{P_\emptyset}\P_{P_i}(\EE)) = 
\bigoplus_{\substack{w\in  W_{P_\emptyset}^{P_i} \\ \l(w) = \pbar_{P_i}(P_\emptyset)}}
H^{\l(w)}(\nn_{P_\emptyset}^{P_i},E) 
=  H^{0}(\nn_{P_\emptyset}^{P_i},E)_{e}
\]
since $\pbar_{P_i}(P_\emptyset) =0$.  Thus we find
\begin{equation}
\Ext^1_{\Perv_-}(\P_{P_i}(\EE),\P_{P_\emptyset}(\EE'))= \begin{cases}
        \CC   & \text{if $E' = \CC_\lambda$,} \\
        0     & \text{otherwise.}
\end{cases}
\end{equation}

\subsection{Results for $T\prec T'$}
\label{ssectResultsTprecTprime}
Let $E'$ be an irreducible algebraic representation of $L_{T'}$ with
highest weight $\lambda$.  Similarly to the above, one may check that for
$T \prec T'$, the only nonzero extensions are as follows:
\begin{align}
\Ext^1_{\Perv_-}(\P_{P_\emptyset}(\EE),\P_{P_1}(\EE'))&= \begin{cases}
        \CC   & \text{if $E = \CC_{s_1(\lambda+\rho)-\rho}$,} \\
        0     & \text{otherwise.}
\end{cases}\\
\Ext^1_{\Perv_-}(\P_{P_\emptyset}(\EE),\P_{P_2}(\EE'))&= \begin{cases}
        \CC   & \text{if $E = \CC_{s_2(\lambda+\rho)-\rho}$,} \\
        0     & \text{otherwise.}
\end{cases}\\
\Ext^1_{\Perv_-}(\P_{P_1}(\EE),\P_{G}(\EE'))&= \begin{cases}
        \CC   & \text{if $E$ has highest weight $s_2s_1(\lambda+\rho)-\rho$,}\\
        0     & \text{otherwise.}
\end{cases}\\
\Ext^1_{\Perv_-}(\P_{P_2}(\EE),\P_{G}(\EE'))&= \begin{cases}
        \CC   & \text{if $E$ has highest weight $s_1s_2(\lambda+\rho)-\rho$,}\\
        0     & \text{otherwise.}
\end{cases}
\end{align}

\bibliography{../references}

\end{document}